\RequirePackage{fix-cm}

\documentclass[smallextended,final]{svjour3} 
\smartqed 

\makeatletter
\def\cl@chapter{}
\makeatother

\journalname{Computational Optimization and Applications}
\spnewtheorem{fact}{Fact}{\bfseries}{\it}
\usepackage[misc]{ifsym}

\usepackage{enumitem}
\usepackage{mathtools}
\usepackage{hyperref}
\usepackage{cleveref}
\usepackage{autonum}
\usepackage{cite}
\usepackage{color}
\usepackage{float}
\usepackage{subcaption}
\usepackage{graphicx}
\definecolor{labelkey}{rgb}{0,0.08,0.45}
\definecolor{refkey}{rgb}{0,0.6,0.0}

\usepackage{url}
\usepackage[output-decimal-marker={.}]{siunitx}
\usepackage{tabularx,booktabs}
\usepackage[left,mathlines]{lineno}

\usepackage{amssymb}
\usepackage{fancyhdr}

\usepackage[margin=1.0in]{geometry}

\newcommand{\menge}[2]{\big\{{#1}~\big |~{#2}\big\}}

\newcommand{\fenv}[1]%
{\ensuremath{\,\overrightarrow{\operatorname{env}}_{#1}}}
\newcommand{\benv}[1]%
{\ensuremath{\,\overleftarrow{\operatorname{env}}_{#1}}}

\newcommand{\scal}[2]{\left\langle{#1},{#2}  \right\rangle}

\newcommand{\RR}{\ensuremath{\mathbb R}}

\newcommand{\NN}{\ensuremath{\mathbb N}}

\newcommand{\lspan}{\ensuremath{\operatorname{span}}}
\newcommand{\aff}{\ensuremath{\operatorname{aff}}}

\newcommand{\Id}{\ensuremath{\operatorname{Id}}}

\newcommand{\tol}{\ensuremath{\mathtt{tol}}}

\DeclareMathOperator{\dist}{dist}

{\begin{list}{}{%
\settowidth{\labelwidth}{\textrm{#1~}}%
\setlength{\leftmargin}{\labelwidth+\labelsep}}}
{\end{list}}

\renewcommand\theenumi{(\roman{enumi})}
\renewcommand\theenumii{(\alph{enumii})}

\newcounter{count}

\allowdisplaybreaks

\begin{document}

\title{The Block-wise Circumcentered--Reflection Method\thanks{RB was partially supported by  Brazilian Agency  \emph{Conselho Nacional de Pesquisa} (CNPq), Grants 304392/2018-9 and 429915/2018-7; JYBC was partially supported by the \emph{National Science Foundation} (NSF), Grant DMS -- 1816449.}}
\author{Roger Behling\and 
J.-Yunier Bello-Cruz \and Luiz-Rafael Santos}


\institute{
  Roger Behling \Letter   \at School of Applied Mathematics, Funda\c{c}\~ao Get\'ulio Vargas \\ 
Rio de Janeiro, RJ -- 22250-900, Brazil. \email{rogerbehling@gmail.com}
    \and 
J.-Yunier Bello-Cruz \at Department of Mathematical Sciences, Northern Illinois University. \\  DeKalb, IL -- 60115-2828, USA. \email{yunierbello@niu.edu}
\and 
 Luiz-Rafael Santos \at Department of Mathematics, Federal University of Santa Catarina. \\ 
Blumenau, SC -- 88040-900, Brazil. \email{l.r.santos@ufsc.br}
}

\maketitle

\begin{abstract} 
The elementary Euclidean concept of circumcenter has recently been employed to improve two aspects of the classical Douglas--Rachford method for projecting onto the intersection of affine subspaces. The so-called circumcentered--reflection method is able to both accelerate the average reflection scheme by the Douglas--Rachford method and cope with the intersection of more than two affine subspaces. We now introduce the technique of circumcentering in blocks, which, more than just an option over the basic algorithm of circumcenters, turns out to be an elegant manner of generalizing the method of alternating projections. Linear convergence for this novel block-wise circumcenter framework is derived and illustrated numerically. Furthermore, we prove that the original circumcentered--reflection method essentially finds the best approximation solution in one single step if the given affine subspaces are hyperplanes.



\end{abstract}

\section{Introduction}\label{sec:intro}

We consider the important feasibility problem of projecting onto the intersection of affine subspaces, frequently also referred to as \emph{best approximation problem}. Let $\{U_i\}_{i\in \mathbf{I}}$ be a family of finitely many affine subspaces in $\RR^n$ with $\mathbf{I}  \coloneqq \{1,2,\ldots,m\}$ and $m$ fixed (we require no relation between $n$ and $m$). The intersection of the family is denoted by $S\coloneqq \bigcap_{i\in \mathbf{I}} U_i$ (which we assume nonempty) and the problem we are interested in consists of projecting a given point $z\in\RR^n$ onto $S$. Equivalently,
\begin{equation}\label{BestProblem}
\min_{s\in S}\|z-s\|.
\end{equation}
Here and throughout the paper,  $\scal{\cdot}{\cdot}$ stands for the Euclidean inner product and  $\|\cdot\|$ is  the induced norm. The best approximation problem \eqref{BestProblem} has the unique solution $P_S(z)$, where $P_S$ denotes the Euclidean projection onto $S$. Problem \eqref{BestProblem} can, of course, be rewritten as a convex quadratic program with objective function $\frac{1}{2}\|z-s\|^2$ and equality constraints, as each $U_i$ is an affine subspace. Note that $S$ itself is an affine subspace. Note also that the classical problem of finding the least-norm solution of a system of linear equations is a particular case of \eqref{BestProblem}.

Reflection and projection type methods are celebrated tools for solving a variety of feasibility problems, including \eqref{BestProblem}, and they remain trendy due to their balance between good performance and simplicity (see, \emph{e.g.}, \cite{Bauschke:2006ej}). 
Probably the two most famous and standard among these methods are the Douglas--Rachford method (DRM), or averaged alternating reflection method (see, \emph{e.g.}, \cite{BCNPW14}); and the method of alternating projections (MAP) which is also known as von Neumann's {or Kaczmarz' algorithm  (see, \emph{e.g.}, \cite{Bauschke:1993dd,Strohmer:2008cm})}. Upon our ideas presented in \cite{Behling:2017da,Behling:2017bz}, we devote this work to study a circumcenter type method related to both DRM and MAP.  

Suitable DRM and MAP schemes determine the solution of the best approximation problem \eqref{BestProblem}. DRM and MAP only use knowledge provided by projections onto individual sets, which often leads to a desirable low computational cost per iteration. Nonetheless, slow convergence due to zig-zag or spiral behavior are usually inherent to these classical methods (see, \emph{e.g.}, \cite{BCNPW14,BCNPW15, Bauschke:2003gb}). In order to minimize spiralness of Douglas--Rachford sequences to a certain extent, we have introduced the circumcentered--reflection method (CRM). This was firstly done in \cite{Behling:2017da} for problem \eqref{BestProblem} with two sets, that is, $m=2$. In this case, if we have a current iterate say $z\in\RR^n$, DRM moves us to $z_{DRM}\coloneqq\frac{1}{2}(\Id+R_{U_2}R_{U_1})(z)$, whereas MAP provides $z_{MAP}\coloneqq P_{U_2}P_{U_1}(z)$. The symbol $\Id$ denotes the identity operator and $R_{U_i}\coloneqq2P_{U_i}-\Id$ is the reflection operator onto $U_i$. We proposed the iteration $z_{CRM}\coloneqq \textrm{circumcenter}\{z,R_{U_1}(z),R_{U_2}R_{U_1}(z)\}$, where $z_{CRM}$ fulfills two properties: \textbf{(i)} it lies on the affine subspace defined by $z$, $R_{U_1}(z)$ and $R_{U_2}R_{U_1}(z)$, which we denote by $\aff\{z,R_{U_1}(z),R_{U_2}R_{U_1}(z)\}$ and, \textbf{(ii)} $z_{CRM}$ is equidistant to $z$, $R_{U_1}(z)$ and $R_{U_2}R_{U_1}(z)$, therefore the use of the term \emph{circumcenter}. 

The resulting algorithm significantly outperforms DRM and MAP numerically as presented in \cite{Behling:2017da}. This numerical performance of CRM, together with the deficiency of DRM in dealing with more than two sets (see   \cite[Example $2.1$]{Artacho:2014uo} and some modifications \cite{Borwein:2014ka,Borwein:2015vm} for DRM), motivated our theoretical study in \cite{Behling:2017bz}. The circumcenter schemes we came up with are already in the attention of specialists of the field (see Arag{\'o}n Artacho \emph{et al} \cite{AragonArtacho:2019ug}, Bauschke \emph{et al} \cite{Bauschke:2018ut,Bauschke:2018wa,Bauschke:2019uh}, 
Lindstrom and Sims \cite{Lindstrom:2018uc} and Ouyang \cite{Ouyang:2018gu}) and questions on the possibility of successful behavior in more general and more important settings are arising. It is worth emphasizing that DRM handles satisfactorily some highly relevant kinds of problems related to nonconvex and inconsistent feasibility problems involving (affine) subspaces (see, for instance,  \cite{Bauschke:2016jw,Demanet:2016fj,Hesse:2013cv,Hesse:2014gi,Bauschke:2013jb}). This suggests a promising behavior of circumcenter-type methods for these kinds of problems since CRM may be seen as a geometrical improvement of DRM.

The linear convergence of the circumcentered--reflection method (CRM) was established in \cite{Behling:2017bz} for solving problem \eqref{BestProblem} with $m\geq 2$ affine subspaces. Since the computation of a circumcenter requires the resolution of a suitable $m\times m$ linear system, this might not be of negligible computational cost for large $m$. To avoid this drawback for problems where the computation of $z_{CRM}$  is simply too demanding,  we propose in the present work the Block-wise Circumcentered--Reflection Method (Bw-CRM) by using an arbitrary ordered partition of the indices $\{1,2,\ldots,m\}$, which contains the original CRM described above as a particular realization. Moreover, two elegant connections of this scheme with MAP follow. These nice interpretations further indicate a possible potential of the proposed method for solving problems more general than \eqref{BestProblem} (even nonconvex), where some affine structure remains, though.

The presentation of this paper is as follows. Definitions, basic facts and important auxiliary results are presented in Section \ref{sec:preliminary}. Still in Section \ref{sec:preliminary}, we introduce the notion of \emph{best approximation mapping} along with properties of these mappings, which are key to our work. In Section \ref{sec:bwCRM}, we formally introduce Bw-CRM. The global Q-Linear convergence of Bw-CRM for problem \eqref{BestProblem} is proven in Section \ref{sec:bwCRM-linearconv}. Connections between Bw-CRM and MAP are briefly discussed in Section \ref{sec:bwCRM-MAP}. In Section \ref{sec:onestepCRM-hyperplane} we prove the curious CRM feature of solving problem \eqref{BestProblem} in only one step when the correspondent affine subspaces are hyperplanes. Numerical illustrations are presented in Section \ref{sec:numerical-llustrations}. In Section \ref{sec:conlcuding} we provide a summary of our work and new ideas for future investigation.

\section{Preliminary and auxiliary results}\label{sec:preliminary}
Let us review the definition of Friedrichs angle and provide key results needed in sequel.

\begin{definition}[Friedrichs angle]
The \emph{cosine of the Friedrichs angle}
between   affine subspaces $V$ and $W$ with nonempty intersection is given by
\[
\begin{aligned}
c_F &\coloneqq \sup \menge{\scal{v}{w}}{v\in \hat V\cap (\hat V\cap
\hat W)^\perp,\; w\in \hat W\cap (\hat V\cap \hat W)^\perp,\;
\|v\|\leq 1,\; \|w\|\leq 1}.
\end{aligned}
\]
Here, $\hat V$ and $\hat W $ are subspaces given by  $V - \hat z$ and $W - \hat z$, respectively, where  $\hat z\in V\cap W$ is arbitrary but fixed, and  the $\perp$ operation provides the correspondent orthogonal subspace.
\end{definition}

In the above definition, it is easy to check that $c_F$ does not depend on the choice of $\hat z$. Moreover, it is well known that  $0\le c_F<1$, for any two intersecting affine subspaces. See fundamental properties of the Friedrichs angle in \cite[Theorem~13]{Deutsch:1995ja} and \cite[Lemma~9.5]{Deutsch:2001fl}, for instance.

For clearer presentation of our results, we introduce the concept of \emph{best approximation mapping} (BAM).

\begin{definition}[best approximation mapping]\label{def:approxproj} Let $V\neq \varnothing$ be a given  affine subspace in $\RR^n$. We say that  $G_V:\RR^n\to\RR^n$ is a \emph{best approximation mapping} with respect to $V$ (for short $V$-BAM) if 
\item[   (i)] $P_V(G_V(z)) = P_V(z)$, for all $z \in\RR^n$; and 
\item[   (ii)]  there exists a constant $r_V\in[0,1)$ such that $
\|G_V(z)-P_V(z)\|\leq r_V \|z-P_V(z)\|$,  for all $z \in\RR^n$.
\end{definition}
Note that the projection operator $P_V$  is a $V$-BAM. Indeed,  
if $G_V = P_V$, for any $z \in\RR^n$ and all $r_V\geq 0$,  we have $P_V(G_V(z)) = P_V(P_V(z)) = P_V(z)$ and   $\|G_V(z)-P_V(z)\| = \|P_V(z)-P_V(z)\| = 0 \leq r_V \|z-P_V(z)\|$. In general, it is easy to see that  $G_V =  (1-\alpha)\Id+\alpha P_V$ with $0<\alpha<2$ is a $V$-BAM  with $r_V=|1-\alpha|\in [0,1)$. Nonetheless,  Definition \ref{def:approxproj} allows for non-affine mappings. Later we will see and use the fact that the circumcenter operator defined in \cite{Behling:2017bz} is a \emph{best approximation mapping}, even though it is usually non-affine.

Simple manipulations provide an immediate consequence of \Cref{def:approxproj}.

\begin{proposition} \label{prop:BAMsequence} Let $G_V$ be a $V$-BAM with constant  $r_V\in[0,1)$.
For any $z\in \RR^n$ and $\ell\in\NN$,  $P_{V}(G_V^\ell(z))=P_{V}(z)$ and $(G_V^k(z))_{k\in\NN}$ converges to $P_{V}(z)$ with linear rate $r_{V}$. 
\end{proposition}

\begin{proof}
Let $z\in\RR^n$ and $\ell\in\NN$. Then, Definition \ref{def:approxproj}(i) implies
$$P_{V}(G_V^\ell(z))=P_{V}(G_V(G_V^{\ell-1}(z)))=P_{V}(G_V^{\ell-1}(z))=\cdots=P_{V}(G_V(z))=P_{V}(z).$$
Moreover,
\begin{align}
\|G^k_V(z)-P_V(z)\| &= \|G_V(G_V^{k-1}(z))-P_V(G^{k-1}_V(z))\| \\& \leq r_V \|G^{k-1}_V(z)-P_V(G^{k-1}_V(z))\|\le \cdots \le r_V^{k-1} \|G_V(z)-P_V(G_V(z))\|\\ 
&= r_V^{k-1} \|G_V(z)-P_V(z)\|\\&\le r_V^k \|z-P_V(z)\|,\end{align} proving the proposition.
 \qed
\end{proof}
The main purpose of this section is to study the composition of best approximation mappings. In order to do this, we state and prove an auxiliary result on adjacent angles. 

\begin{proposition}\label{prop:gamma} Let $u, v \in\RR^{n}$ be nonzero vectors forming an angle $\gamma\in[0,\frac{\pi}{2}]$. If a nonzero vector $w\in\RR^{n}$ forms an angle $\beta\in[0,\pi]$ with $v$ and the angle $\phi$ between $w$ and $u$ is such that $\phi\in[0,\frac{\beta}{2}]$, then $\gamma \geq \frac{\beta}{2}$.
\end{proposition}

\begin{proof}
Assume, without loss of generality, that $\|u\|=\|v\|=\|w\|=1$. Then, 
\[ \label{eq:cosines}
\cos\gamma = \scal{u}{v},\; \cos\beta = \scal{v}{w}\text{ and } \cos\phi = \scal{u}{w}. 
\]
Also, we have  $\cos\gamma\geq 0$ and $\cos\phi\geq 0$.

If $\beta=0$, $\gamma \geq \frac{\beta}{2}$, trivially. Moreover, $\beta = \pi$ if, and only if, $w = -v$. In this case,
\[
  0 \leq \cos\phi =  \scal{u}{w} =  \scal{u}{-v} =  - \scal{u}{v} = -\cos \gamma \leq 0.
\]
Thus, $\cos\gamma = 0$ and $\gamma  = \frac{\pi}{2} = \frac{\beta}{2}$.

For the rest of the proof, let  $\beta\in(0,\pi)$ and consider  the following convex optimization problem
\[\label{eq:convexoptaux}
\begin{aligned}
\min_{x} &&& -\scal{v}{x} \\
{\text{s.t. } } &&& \frac{1}{2}\|x\|^{2} - \frac{1}{2} \leq 0 \\
              &&& \cos \tfrac{\beta}{2} - \scal{w}{x} \leq 0.
\end{aligned}
\]
 
 By Weierstrass, this problem has a solution as the objective function is continuous  and the feasible set is compact. Note that $\frac{1}{2}\left(1+\cos\tfrac{\beta}{2}\right)w$ is a Slater point, that is, it  fulfills both constraints strictly because $\cos \tfrac{\beta}{2}<1 $ as $\beta\in(0,\pi)$. Therefore, $x$ is a solution of \eqref{eq:convexoptaux}  if, and only if, it satisfies  the KKT conditions
 \begin{subequations}\label{eq:kkt}
\begin{align}
-v + \mu_1 x - \mu_2 w = 0  \label{eq:kkt_auxlemma1} \\
  \|x\| - 1 \leq 0 \label{eq:kkt_auxlemma2} \\
\cos \tfrac{\beta}{2} - \scal{w}{x} \leq 0\label{eq:kkt_auxlemma3} \\
 \mu_1\geq 0, \; \mu_2 \geq 0\label{eq:kkt_auxlemma4} \\
\mu_1\left(\|x\| - 1\right) = 0 \label{eq:kkt_auxlemma5} \\
\mu_2\left(\cos \tfrac{\beta}{2} - \scal{w}{x}\right) = 0, \label{eq:kkt_auxlemma6} \\
 \end{align}
 \end{subequations}
where $\mu_{1}$ and $\mu_{2}$ are Lagrange multipliers.

We claim that $x^*\coloneqq \frac{v+w}{\|v+w\|}$, which is well defined since $\beta \neq \pi$, is a KKT point for  \eqref{eq:convexoptaux}. Condition \eqref{eq:kkt_auxlemma1} is satisfied with $\mu_1^*=\|v+w\|$ and $\mu_2^*=1$. These multipliers yield \eqref{eq:kkt_auxlemma4} strictly. Trivially,  since $\|x^*\|=1$, condition \eqref{eq:kkt_auxlemma2} holds sharply and \eqref{eq:kkt_auxlemma5} follows as well. Obviously,  from $\scal{v}{v}=\scal{w}{w}=1$, we have
\[\label{eq:xstarw-xstarv} \scal{w}{x^*} =
\scal{w}{\dfrac{v+w}{\|v+w\|}} = \dfrac{\scal{w}{v} + \scal{w}{w}}{\|v+w\|} = \dfrac{  \scal{v}{v}+\scal{v}{w}}{\|v+w\|}  =  \scal{v}{\dfrac{v+w}{\|v+w\|}} = \scal{v}{x^*}.\] 
Then,
\begin{align}
2\scal{w}{x^*} &  = 
\scal{v}{x^*}  + \scal{w}{x^*}    =  \scal{v}{\dfrac{v+w}{\|v+w\|}} + \scal{w}{\dfrac{v+w}{\|v+w\|}} \\
  &  =\dfrac{\scal{v+w}{v+w}}{\|v+w\|} = \dfrac{\|v+w\|^2}{\|v+w\|} = \|v+w\|\\
  & = \sqrt{\|v\|^2+\|w\|^2+2\scal{v}{w}} \\ 
  & = \sqrt{2+2\cos\beta} = 2\sqrt{\frac{1+\cos\beta}{2}} \\
  &= 2\cos\frac{\beta}{2}, \label{eq:wcosbeta2} \\ 
 \end{align}
that is, condition \eqref{eq:kkt_auxlemma3} holds sharply and yields  \eqref{eq:kkt_auxlemma6}. 

Note that  $u$ is a feasible point for \eqref{eq:convexoptaux}. In fact, by assumption $\|u\|=1$ and $0\leq \phi\leq \frac{\beta}{2}$, which means that  $\scal{w}{u} = \cos\phi \geq \cos\frac{\beta}{2}$.

Finally, using the definition of $\gamma$, the optimality of $x^*$, \eqref{eq:xstarw-xstarv} and \eqref{eq:wcosbeta2} we derive
\[
\cos\gamma = \scal{v}{u} \leq  \scal{v}{x^*} = \scal{w}{x^*} = \cos\tfrac{\beta}{2}.
\]
Hence, $\gamma\geq \frac{\beta}{2}$,  proving the lemma. \qed
\end{proof}

We now start to address  the composition of  {best approximation mappings}. The next result is the keystone of our analysis.

\begin{lemma}[composition of two best approximation mappings]\label{LemmaAux1} Let us consider two affine subspaces $V$ and $W$  of $\, \mathbb{R}^n$ with nonempty intersection $V\cap W$. Then, the composition of a $V$-BAM  and a $W$-BAM is a $(V\cap W)$-BAM.
\end{lemma}
\begin{proof}
Let  $G_V:\RR^n\to\RR^n$ and $G_W:\RR^n\to\RR^n$ be two \emph{best approximation mappings} with respect to $V$ and $W$, respectively, and corresponding constants  $0\le r_V<1$ and $0\le r_W<1$.

In order to prove item (i)
of Definition \ref{def:approxproj} for the composition $G\coloneqq G_W \circ G_V$ w.r.t. $V \cap W$ we are going to combine Pythagoras equations with properties of projections. 
Note that we have to prove that  $P_{V\cap W}(G(z)) = P_{V\cap W}(z)$ for all $z \in\RR^n$.

Let us take an arbitrary, but fixed, $z\in\RR^n$ and set $\hat z\coloneqq P_{V\cap W}(z)$. The definition of $\hat z$ implies that $\hat z\in V\cap W$. In particular, $\hat z\in V$ and we have 
\begin{equation}\label{Pyth1}
\|z-\hat z\|^2=\|z-P_V(z)\|^2+\|P_V(z)-\hat z\|^2.
\end{equation}
Since $P_{V\cap W}(P_{V}(z))\in V$,  we can write
\begin{equation}\label{Pyth1.A}
\|z-P_{V\cap W}(P_{V}(z))\|^2=\|z-P_V(z)\|^2+\|P_V(z)-P_{V\cap W}(P_{V}(z))\|^2.
\end{equation}
Of course, $\|z-\hat z\|\leq \|z-P_{V\cap W}(P_{V}(z))\|$. Using this fact and subtracting \eqref{Pyth1.A} from \eqref{Pyth1}
yields $\|P_V(z)-\hat z\|\leq \|P_V(z)-P_{V\cap W}(P_{V}(z))\|$. By uniqueness of projections onto closed convex sets, we conclude that 
\[\label{eq:zhat=PVW-PVz}
\hat z=P_{V\cap W}(P_{V}(z)).
\]
The fact that both $\hat z$ and $P_{V\cap W}(G_V(z))$ lie in $V$ allows us to derive further Pythagoras relations
\begin{equation}
\|G_V(z)-\hat z\|^2=\|G_V(z)-P_V(G_V(z))\|^2+\|P_V(G_V(z))-\hat z\|^2
\end{equation}
and
\begin{equation}
\|G_V(z)-P_{V\cap W}(G_V(z))\|^2=\|G_V(z)-P_V(G_V(z))\|^2+\|P_V(G_V(z))-P_{V\cap W}(G_V(z))\|^2.
\end{equation}
Since $G_V$ is a $V$-BAM, it holds that $P_V(G_V(z))=P_V(z)$ and the previous equations reduce to
\begin{equation}\label{Pyth2}
\|G_V(z)-\hat z\|^2=\|G_V(z)-P_V(z)\|^2+\|P_V(z)-\hat z\|^2
\end{equation}
and
\begin{equation}\label{Pyth2.A}
\|G_V(z)-P_{V\cap W}(G_V(z))\|^2=\|G_V(z)-P_V(z)\|^2+\|P_V(z)-P_{V\cap W}(G_V(z))\|^2.
\end{equation}
As proved above in \eqref{eq:zhat=PVW-PVz}, $\hat z=P_{V\cap W}(P_{V}(z))$, which implies that $\|P_V(z)-\hat z\|\leq \|P_V(z)-P_{V\cap W}(G_V(z))\|$. This inequality, together with \eqref{Pyth2} and \eqref{Pyth2.A}, gives us $\|G_V(z)-\hat z\| \leq \|G_V(z)-P_{V\cap W}(G_V(z))\|$. Then, as $P_{V\cap W}(G_V(z))$ is uniquely defined and $\hat z \in  V\cap W$, we must have 

\[\label{eq:projVcapWG_V-zhat}
P_{V\cap W}(G_V(z))=\hat z.
\]

Our proof towards item (i) of Definition \ref{def:approxproj} continues with similar arguments, now regarding $W$. By Pythagoras we get
\begin{equation}
\|G_V(z)-P_{V\cap W}(G_V(z))\|^2=\|G_V(z)-P_W(G_V(z))\|^2+\|P_W(G_V(z))-P_{V\cap W}(G_V(z))\|^2
\end{equation}
and
\begin{equation}
\|G_V(z)-P_{V\cap W}(G_W(G_V(z)))\|^2=\|G_V(z)-P_W(G_V(z))\|^2+\|P_W(G_V(z))-P_{V\cap W}(G_W(G_V(z)))\|^2.
\end{equation}
Since we proved that $P_{V\cap W}(G_V(z))=\hat z$,  taking into account that $G_W$ is a $W$-BAM,  which provides  $ P_W(G_W(G_V(z))) = P_W(G_V(z))$, and bearing in mind that $G(z)=G_W(G_V(z))$, we can rewrite the equations above as
\begin{equation}\label{Pyth.B1}
\|G_V(z)-\hat z\|^2=\|G_V(z)-P_W(G(z))\|^2+\|P_W(G(z))-\hat z\|^2
\end{equation}
and
\begin{equation}\label{Pyth.B2}
\|G_V(z)-P_{V\cap W}(G(z))\|^2=\|G_V(z)-P_W(G(z))\|^2+\|P_W(G(z))-P_{V\cap W}(G(z))\|^2.
\end{equation}
From the definition of Euclidean projection it follows that $ \|G_V(z)-\hat z\|\le \|G_V(z)-P_{V\cap W}(G_V(z))\| $, because $\hat z$ realizes the distance of $G_{V}(z)$ to $V\cap W$. This, combined with \eqref{Pyth.B1} and \eqref{Pyth.B2}, leads to $\|P_W(G(z))-\hat z\|\le \|P_W(G(z))-P_{V\cap W}(G(z))\|$. 

We can derive two additional Pythagoras relations
\begin{equation}\label{Pyth.C1}
\|G(z)-\hat z\|^2=\|G(z)-P_W(G(z))\|^2+\|P_W(G(z))-\hat z\|^2
\end{equation}
and
\begin{equation}\label{Pyth.C2}
\|G(z)-P_{V\cap W}(G(z))\|^2=\|G(z)-P_W(G(z))\|^2+\|P_W(G(z))-P_{V\cap W}(G(z))\|^2.
\end{equation}
We have just seen that $\|P_W(G(z))-\hat z\|\le \|P_W(G(z))-P_{V\cap W}(G(z))\|$, which together with \eqref{Pyth.C1} and \eqref{Pyth.C2} yields $\|G(z)-\hat z\| \leq \|G(z)-P_{V\cap W}(G(z))\|$. Hence, $P_{V\cap W}(G(z)) = \hat z =  P_{V\cap W}(z)$, which fulfills condition (i) of \Cref{def:approxproj} for $G$ w.r.t. $V\cap W$.

Let us now   address item  (ii) of \Cref{def:approxproj} for $G$ w.r.t. $V\cap W$. We have to prove that there exists a nonnegative constant $0\le r_{V\cap W}<1$ so that, for all $z\in\RR^n$, 
\begin{equation} \label{MainIneqLemmaG}
\|G(z)-P_{V\cap W}(z)\|\leq r_{V\cap W} \|z-P_{V\cap W}(z)\|.
\end{equation} 
Again,  let $z\in\RR^n$ be arbitrary,  fixed  and  $\hat z =  P_{V\cap W}(z)$. If $z = \hat z$, \eqref{MainIneqLemmaG} is fulfilled   for any nonnegative constant as $G(z)$ will be equal to $ P_{V\cap W}(z)$. In fact, $G_V$ being a $V$-BAM, together with $z =\hat z $,   gives us
\[
\| G_{V}(z) - \hat z\|  = \| G_{V}(z) - P_{V}(z)\| \leq r_{V}\|z - P_{V}(z) \| = r_{V}\|z - \hat z \| = 0.
\]
Thus, $ G_{V}(z) = \hat z$.  On the other hand, $G_W$ is a $W$-BAM, so
\[
\| G(z) -  z\| = \| G(z) - \hat z\|  = \| G_W(G_V(z)) - \hat z\| =    \| G_W(\hat z) - \hat z\|  = \| G_W(\hat z) - P_{W}(\hat z)\| \leq r_{W}  \|\hat z - P_{W}(\hat z) \| = 0.
\]
This means that, if $z=\hat z$,  $G(z) =  z$ and the left-hand side of \eqref{MainIneqLemmaG} is equal to zero and  this inequality  holds for  any nonnegative constant $r_{V\cap W}$. 

Therefore, from now on, assume $z\neq \hat z$.  We will construct  $r_{V\cap W}$  upon the constants $r_V\in[0,1)$, $r_W\in[0,1)$ , and $c_F\in[0,1)$, the latter the cosine of the Friedrichs angle $\theta_F\in(0,\frac{\pi}{2}]$ between $V$ and $W$.

It will be key to look at the angle $\alpha$ between vectors $z-\hat z$ and $P_V(z) - \hat z$. Note first that $\alpha\in [0,\frac{\pi}{2}]$, since from  \eqref{Pyth1} the triangle with vertices  $z$, $\hat z$ and $P_V(z)$ has a right angle at $P_V(z)$. Also, of course, 
\[\label{eq:cosalpha}
\cos \alpha = \dfrac{\|P_V(z) - \hat z\|}{\|z-\hat z\|}.
\]
Moreover, by using equation \eqref{Pyth2}, the $V$-BAM hypothesis  $\|G_V(z)-P_V(z)\|\leq r_V \|z-P_V(z)\|$, equation \eqref{Pyth1} and that $\cos \alpha \leq 1$,  we conclude that 
\begin{align}                      
  \|G_{V}(z)-\hat z\|^2   & =\|G_V(z)-P_V(z)\|^2+\|P_V(z)-\hat z\|^2 \notag \\ 
                    & \leq r_V^2\|z-P_V(z)\|^2+\|P_V(z)-\hat z\|^2 \notag \\
                   &  = r_V^2\left(\|z-\hat z\|^2-\|P_V(z)-\hat z\|^2 \right)+\|P_V(z)-\hat z\|^2 \notag \\
                  & = r_V^2\|z-\hat z\|^2+(1-r_V^2)\|P_V(z)-\hat z\|^2  \notag \\      
                  &  =  r_V^2\|z-\hat z\|^2+(1-r_V^2)\cos^{2}\alpha\|z-\hat z\|^2  \label{eq:GV_cosalpha}  \\ 
                  & \leq   \|z-\hat z\|^2. \label{eq:GV_z-zhat}               
\end{align}
Now, we split our analysis in two cases: $\alpha\in [\frac{\theta_F}{2},\frac{\pi}{2}]$;    $\alpha\in [0,\frac{\theta_F}{2})$.

\noindent \textbf{Case 1:} $\alpha\in [\frac{\theta_F}{2},\frac{\pi}{2}]$. 

In this case, $ \cos \alpha  \leq \cos \frac{\theta_{F}}{2}$. This, combined with \eqref{eq:GV_cosalpha} and the fact that  $\cos \frac{\theta_{F}}{2} = \sqrt{\frac{1+c_F}{2}}$  provides 
\begin{align}                      
  \|G_{V}(z)-\hat z\|^2  & \leq  r_V^2\|z-\hat z\|^2+(1-r_V^2)\cos^{2}\alpha\|z-\hat z\|^2  \notag \\ 
                         & \leq r_V^2\|z-\hat z\|^2+(1-r_V^2)\cos^{2}\frac{\theta_{F}}{2}\|z-\hat z\|^2  \notag \\  
                         & = \left(r_V^2+(1-r_V^2)\frac{1+c_F}{2}\right)\|z-\hat z\|^2. \label{eq:r_1.0}               
\end{align}
Since $r_V^2<1$ and $\frac{1+c_F}{2}<1$,  we have $(1-r_V^2)\frac{1+c_F}{2} <  (1-r_V^2)$. 
Then,
\[
r_V^2+(1-r_V^2)\frac{1+c_F}{2} < r_V^2 + (1-r_V^2) =  1.
\]
Since $G_W$ is a $W$-BAM, we have $ P_W(G_W(G_V(z))) = P_W(G_V(z))$ and  $\|G_W(G_V(z))-P_W(G_V(z))\|  \leq r_{W}\|G_V(z)-P_W(G_V(z))\|$, with $r_{W}\in [0,1)$. So,  we can write $\|G(z)-P_W(G(z))\| \leq \|G_V(z)-P_W(G(z))\|$, which combined  with  \eqref{Pyth.B1} and \eqref{Pyth.C1}, gives us 
\[
\|G(z)-\hat z\| \leq \|G_V(z)-\hat z\|.
\]
Hence, this inequality and \eqref{eq:r_1.0} imply that 
\[
\label{eq:Gr_1}
\|G(z)-\hat z\| \leq r_1\|z-\hat z\|, 
\]
with $r_1\in [0,1)$, given by 
\[\label{eq:r_1}
r_1 \coloneqq \sqrt{r_V^2+(1-r_V^2)\frac{1+c_F}{2}}.
\]

\noindent \textbf{Case 2:} $\alpha\in [0,\frac{\theta_F}{2})$.

In this case we initially consider the triangle of vertexes  $G_V(z)$, $\hat z$ and $P_W(G_V(z))$.  Since $G_{V}$ is a $V$-BAM, $P_V(G_V(z)) = P_V(z)$.  We will be particularly interested in the angle  $\phi$  between $G_V(z)-\hat z$ and $P_V(z)-\hat z$, when these vectors are nonzero. The vector $P_V(z)-\hat z$ is automatically nonzero, because of $\alpha < \frac{\pi}{2}$  and   \eqref{eq:cosalpha}. If the  vector $G_V(z)-\hat z$ is zero, we get the desired result as shown below.

Suppose  $G_V(z) = \hat z$, then $G(z) = G_W(G_V(z)) =  G_W(\hat z)$ and it is easy to verify that $G_W(\hat z) = \hat z$. Indeed, 
\[
\| G_W(\hat z) -  \hat z\| =\| G_W(\hat z) -  P_W(\hat z)\| \leq r_{W} \| \hat z -  P_W(\hat z)\| =     \| \hat z - \hat z\|  =  0.
\]
So, $G(z)  = \hat z$ and the left-hand side of \eqref{MainIneqLemmaG} is equal to zero and  this inequality is fulfilled  for  any nonnegative constant $r_{V\cap W}$. 

Assume for the rest of the proof that $G_V(z) \neq \hat z$.  Thus, 
\[
\cos\phi = \frac{\|P_V(z)-\hat z\|}{\|G_V(z)-\hat z\|} \geq \frac{\|P_V(z)-\hat z\|}{\|z-\hat z\|} = \cos \alpha,
\]
 where the inequality is due to  \eqref{eq:GV_z-zhat}. Therefore, $0\leq \phi\leq \alpha$ and, consequently, $\phi\in [0,\frac{\theta_{F}}{2})$.

We consider now another triangle, the one of vertexes $P_V(z)$, $\hat z$ and $P_W(G_V(z))$. If the vertexes $\hat z$ and $P_W(G_V(z))$ coincide, we get the following bound:
\begin{align}
\| G(z) -  \hat z\| & = \| G_W(G_V(z)) -  \hat z\| \\ 
                    & = \| G_W(G_V(z)) -  P_W(G_V(z)) \| \\ 
                    & \leq r_{W}\|  G_V(z) -  P_W(G_V(z)) \| \\
                    & \leq r_{W}\|  G_V(z) -  \hat z \| \\ 
                    & \leq r_{W}\| z - \hat z \|, \label{eq:G_rW}
\end{align}
where we used, respectively, the definition of $G$, the current assumption  $P_W(G_V(z))=\hat z$,  the hypothesis that $G_W$ is a $W$-BAM, the fact that $\hat z$ lies in $W$  and \eqref{eq:GV_z-zhat}.

For the rest of the proof, assume also that $P_W(G_V(z))\neq \hat z$ and define  $\beta$, the angle between  the nonzero vectors $P_V(z)-\hat z$ and $P_W(G_V(z))-\hat z$. It is easy to see that the former belongs to $\hat V$ and the latter belongs to $\hat W$, where $\hat V$ and $\hat W$ are the subspaces given by $ V - \hat z$ and $  W - \hat z$, respectively. Also, recall from \eqref{eq:zhat=PVW-PVz} that $P_{V\cap W}(P_V(z)) = \hat z$ and therefore  $ P_V(z)-\hat z \in (\hat V \cap \hat  W)^{\perp}$.

We  rewrite \eqref{Pyth.B1} using the $W$-BAM property $P_W(G(z)) = P_W(G_W(G_V(z)))= P_W(G_V(z))$ as 
\[\label{Pyth.B1-rewritten}
\|G_V(z)-\hat z\|^2=\|G_V(z)-P_W(G_V(z))\|^2+\|P_W(G_V(z))-\hat z\|^2
\]
and Pythagoras can be employed as
\[
\|G_V(z)-P_{V\cap W}(P_W(G_V(z)))\|^2=\|G_V(z)-P_W(G_V(z))\|^2+\|P_W(G_V(z))-P_{V\cap W}(P_W(G_V(z)))\|^2.
\]
On the one hand, $\|P_W(G_V(z))-P_{V\cap W}(P_W(G_V(z)))\|\leq\|P_W(G_V(z))-\hat z\| $. On the other hand, $\|G_V(z)-\hat z\|\leq \|G_V(z)-P_{V\cap W}(P_W(G_V(z)))\| $ because  we have already seen in \eqref{eq:projVcapWG_V-zhat} that $P_{V\cap W}(G_V(z)) = \hat z$. Hence, $P_{V\cap W}(P_W(G_V(z))) = \hat z$ and $ P_W(G_V(z))-\hat z \in (\hat V \cap \hat  W)^{\perp}$.

We can then  use the definition of the cosine of the Friedrichs angle $\theta_{F}$ between $V$ and $W$ and get
\[
 \cos\beta=\scal{\frac{P_W(G_V(z))-\hat z} {\|P_W(G_V(z))-\hat z\|}}{\frac{P_V(z)-\hat z} { \|P_V(z)-\hat z\|}}\leq c_F = \cos\theta_{F},
\]
which  provides $ \beta \in [\theta_{F},\pi] $.

 By now we have the nonzero vectors $G_V(z) - \hat z$, $P_{V}(z) - \hat z$ and $P_W(G_V(z))-\hat z$. The vectors $G_V(z) - \hat z$ and $P_{V}(z) - \hat z$ form  angle $\phi\in [0,\frac{\theta_F}{2}) $,  vectors $P_{V}(z) - \hat z$ and $P_W(G_V(z))-\hat z$ form angle $\beta \in [\theta_{F},\pi]$. Let $\gamma$ be the angle between vectors $G_V(z) - \hat z$ and  $P_W(G_V(z))-\hat z$. Obviously, by Pythagoras, $\gamma \in [0, \frac{\pi}{2}]$ and 
 \[\label{eq:cosinegamma}
\cos \gamma  = \frac{\|P_W(G_V(z))-\hat z\|}{\|G_V(z)-\hat z\|}.
 \] 
 More than that, from \Cref{prop:gamma}, we conclude that $\gamma \in [\frac{\beta}{2},\frac{\pi}{2}]$. In particular, we get $\gamma \geq \frac{\beta}{2} \geq  \frac{\theta_{F}}{2}$ and hence  
 \[\label{eq:boundcosinegamma}
  \cos \gamma \leq \cos \frac{\theta_{F}}{2}  = \sqrt{\frac{1+c_F}{2}}.
 \]
Then, enforcing similar arguments as in Case 1, we obtain
\begin{align}
\|G(z)-\hat z\|^2      & = \|G_W(G_V(z))-\hat z\|^2 \\
                       & = \|G_W(G_V(z))-P_W(G_V(z))\|^2+\|P_W(G_V(z))-\hat z\|^2 \\
                       & \leq r_W^2\|G_V(z)-P_W(G_V(z))\|^2+\|P_W(G_V(z))-\hat z\|^2 \\ 
                       & = r_W^2(\|G_V(z)-P_W(G_V(z))\|^2+\|P_W(G_V(z))-\hat z\|^2)+(1-r_W^2)\|P_W(G_V(z))-\hat z\|^2 \\
                       & = r_W^2\|G_V(z)-\hat z\|^2+(1-r_W^2)\|P_W(G_V(z))-\hat z\|^2 \\
                       & = r_W^2\|G_V(z)-\hat z\|^2+(1-r_W^2)\cos^2\gamma \|G_V(z)-\hat z\|^2 \\ 
                       & \leq \left(r_W^2+(1-r_W^2)\frac{1+c_F}{2}\right)\|G_V(z)-\hat z\|^2 \\
                       & \leq \left(r_W^2+(1-r_W^2)\frac{1+c_F}{2}\right)\|z-\hat z\|^2.\label{eq:G_r2}
\end{align}
The first line corresponds to the definition of $G$, the second is by Pythagoras and the third holds because $G_W$ is a $W$-BAM. The fourth line is a rearrangement of terms, followed by Pythagoras in the fifth. Then, \eqref{eq:cosinegamma} and \eqref{eq:boundcosinegamma} are  employed respectively. At last, we used \eqref{eq:GV_z-zhat}.

Analogously to the proof that $r_1 = \sqrt{r_V^2+(1-r_V^2)\frac{1+c_F}{2}} $ is strictly smaller than $1$, we can see that  
\[
r_2 \coloneqq \sqrt{r_W^2+(1-r_W^2)\frac{1+c_F}{2}} <1 .
\]
Finally, we can gather Cases 1 and 2.  From   \eqref{eq:r_1}, \eqref{eq:G_rW} and  \eqref{eq:G_r2},   we have $\|G(z)-\hat z\| \leq  r_{V\cap W}\|z-\hat z\|$ for all $z\in\RR^n$, with $r_{V\cap W}\in[0,1)$ given by 
\(r_{V\cap W}\coloneqq \max\left\{r_1, r_W, r_2\right\} = \max\left\{r_1, r_2\right\}.
\)\qed
\end{proof}

We are going to see next that  Lemma \ref{LemmaAux1} can be extended to the case of $\ell$ affine subspaces, with $\ell$ being any positive integer.

\begin{theorem}[finite composition of best approximation mappings]\label{thm:BAM} Let us consider an indexed family of $\ell$ affine subspaces $\mathbf{W}=\{W_1,W_2,\dots,W_{\ell}\}$ of $\RR^n$ with nonempty intersection $S_\ell$. Assume that each $G_{W_j}:\RR^n\to\RR^n$ ($j=1,\ldots,\ell$) is $W_j$-BAM. Then, $G\coloneqq G_{W_\ell}\circ \cdots \circ G_{W_2} \circ G_{W_1}$ is a $S_{\ell}$-BAM. 
\end{theorem}
\begin{proof} The proof follows by an induction argument on $\ell$, the number of affine subspaces.

If $\ell=1$, we have $G = G_{W_{1}}$ and then $G$ is a $S_\ell$-BAM.

Assume the result for a fixed $\ell$.
Let $  \widehat{\mathbf{W}} \coloneqq \mathbf{W} \cup \{W_{\ell+1}\}$, where  $W_{\ell+1}$ is an affine subspace such that it has  nonempty intersection $S_{\ell+1}$ with $S_\ell$, and let $G_{W_{\ell+1}}$ be a $W_{\ell+1}$-BAM. 
 Employing  \Cref{LemmaAux1} with $S_\ell$ and $W_{\ell+1}$ playing the role of $V$ and $W$, respectively, and $\widehat G$ and $G_{W_{\ell+1}}$ playing the role of $G_V$ and $G_W$, respectively, we get that $\widehat G \coloneqq  G\circ  G_{W_{\ell+1}}$  is a $S_{\ell+1}$-BAM.
\qed\end{proof}

In the next section, we define the block-wise circumcenter operator and will prove that it is a \emph{best approximation mapping}.

\section{The block-wise circumcentered--reflection method}
\label{sec:bwCRM}
The main purpose of this paper is applying the recently developed circumcentered--reflection method (CRM) \cite{Behling:2017bz} to solve problem \eqref{BestProblem} by taking advantage of a block-wise structure. This idea may be beneficial in certain problems coming from the discretization of partial differential equations as we describe and illustrate in our numerical section.  We remind that CRM iterates by taking an ordered round of successive reflections onto affine subspaces and then it chooses the new iterate by means of equidistance to the reflected points, which explains the usage of the geometric term \emph{circumcenter}.

Let us give the definition of the circumcenter of a  block of finitely many affine subspaces. 

\begin{definition}[circumcentered-reflection for a block]\label{CircDef}
Let $\mathcal{B}\coloneqq(U_1,U_2,\ldots,U_q)$ be a block of ordered affine  subspaces, where $q\ge 1$ is a fixed integer. Suppose also that the intersection $S_\mathcal{B}\coloneqq \bigcap_{i=1}^q U_i$ is nonempty. 
The circumcenter of the block $\mathcal{B}$ at the point $z\in \RR^n$  is denoted by $C_\mathcal{B}(z)$ and defined by the following properties:
	\item [ (i)] $C_\mathcal{B}(z)\in W_z\coloneqq\aff\{z,R_{U_1}(z),R_{U_2}R_{U_1}(z),\ldots ,R_{U_q}\cdots R_{U_2}R_{U_1}(z)\}$;
	\item [ (ii)]$\|z-C_\mathcal{B}(z)\|=\|R_{U_1}(z)-C_\mathcal{B}(z)\|=\cdots=\|R_{U_q}\cdots R_{U_2}R_{U_1}(z)-C_\mathcal{B}(z)\|$.
\end{definition}

It is worth noting that the order in which reflections are composed affects the outcome circumcenter. If not said otherwise, we use  increasing order of indices for the computation of a circumcenter.

Before presenting the definition of the block-wise circumcentered--reflection method (Bw-CRM), we list two consequences of results from Lemma 3.1 of \cite{Behling:2017bz} that will be at the core of our convergence analysis for Bw-CRM.

\begin{lemma}[good definition of CRM]
Consider a block of affine subspaces $\mathcal{B}=(U_1,U_2,\ldots,U_q)$ with $S_\mathcal{B}=\cap_{i=1}^q U_i$ nonempty. For any $z\in\RR^n$, $C_\mathcal{B}(z)$ is well and uniquely defined.
\end{lemma}
\begin{proof}
See \cite[Lemma 3.1]{Behling:2017bz}.
\qed\end{proof}

The circumcenter,  as above, is the intersection of suitable bisectors. Its computation  requires the resolution of a $q\times q$ linear system of equations. Details can be found in \cite[p.~161]{Behling:2017bz} and \cite[Theorem~4.1]{Bauschke:2018ut}.

The previous lemma established that  the circumcenter is well defined. We now recall that the circumcenter operator is a BAM.

\begin{theorem}[circumcenter operator is a BAM]\label{Teo2.3} Consider a block of affine subspaces $\mathcal{B}=(U_1,U_2,\ldots,U_q)$ with $S_\mathcal{B}=\cap_{i=1}^q U_i$ nonempty. Then, there exists a constant $r_\mathcal{B}\in [0,1)$ so that 
\begin{equation}
\|C_\mathcal{B}(z)-P_{S_\mathcal{B}}(z)\|\leq r_\mathcal{B} \|z-P_{S_\mathcal{B}}(z)\|,
\end{equation}
for all $z\in\RR^n$. Moreover, $P_{S_\mathcal{B}}(C_\mathcal{B}(z))=P_{S_\mathcal{B}}(z)$.
\end{theorem}
\begin{proof}
See \cite[Lemma 3.2]{Behling:2017bz}.
\qed\end{proof}

The previous theorem says that  $C_\mathcal{B}$ is a  $S_\mathcal{B}$-BAM.
In order to define our new circumcenter scheme, consider the following terminology. 

\begin{definition}[block partition] \label{BlockDef} We say that $\mathbf{B}=(\mathcal{B}_1,\mathcal{B}_2,\ldots,\mathcal{B}_p)$ is an ordered collection of blocks   (with cardinality $p$) for the ordered affine subspaces $U_1,U_2,\ldots,U_m$ if we can write $\mathcal{B}_1=(U_{q_0+1},U_{q_0+2},\ldots,U_{q_1})$, $\mathcal{B}_2=(U_{q_1+1},U_{q_1+2},\ldots,U_{q_2})$, $\ldots$, $\mathcal{B}_p=(U_{q_{p-1}+1},U_{q_{p-1}+2},\ldots,U_{q_p})$, with $q_0=0$ and $q_p=m$. We assume that every block $\mathcal{B}_i$ has size $q_i-q_{i-1}\geq 1$, $i=1,\ldots,p$.
\end{definition}
Note that in the previous definition we are simply selecting subsets of subspaces based on a partition of the set of indices, illustrated below 
\[
\mathbf{I}=\{1,2,\ldots,m\} = \{\underbrace{q_0+1,\ldots,q_1}_{\text{1st block indices}},\, \underbrace{q_1+1,\ldots,q_2}_{\text{2nd block indices}},\, \underbrace{q_2+1,\ldots,q_3}_{\text{3rd block indices}},\,\ldots, \, \underbrace{q_{p-1}+1,\ldots,q_p}_{\text{$p$-th block indices}}\}.
\]
We now define the block-wise circumcentered-reflection operator.
\begin{definition}[block-wise circumcentered-reflection] \label{def:bwcrm}Let $\mathbf{B}=(\mathcal{B}_1,\mathcal{B}_2,\ldots,\mathcal{B}_p)$ be an ordered collection of  blocks for the affine subspaces $U_1,U_2,\ldots,U_m$ and assume that increasing index order is taken for both blocks and subspaces. Then, for a point $z\in\RR^n$ we define the \emph{block-wise circumcentered-reflection} step $C_{\mathbf{B}}(z)$ by
\[\label{eq:BlockCircDef}
C_{\mathbf{B}}(z)\coloneqq C_{\mathcal{B}_p}\circ C_{\mathcal{B}_{p-1}}\circ \cdots \circ C_{\mathcal{B}_2}\circ C_{\mathcal{B}_1}(z).
\]
\end{definition}

A key result is presented in Section 3.1. It establishes linear convergence of the sequence $\left(C^k_{\mathbf{B}}(z)\right)_{k\in\NN}$ to $P_S(z)$. Our proof that Bw-CRM provides a sequence converging linearly to the solution of the best approximation problem \eqref{BestProblem} depends on some further auxiliary results, derived in the next section.

In the following section, the circumcenter operators for each block $C_{\mathcal{B}_j}$  will play the role of the \emph{best approximation mappings} $G_{W_j}$'s. Furthermore, $C_\mathbf{B}$ will play the role of $G$ in \Cref{thm:BAM}.

\subsection{Linear convergence of the block-wise circumcentered--reflection method}
\label{sec:bwCRM-linearconv}
Now, we summarize our result on Bw-CRM. Remind that $\mathbf{B}=(\mathcal{B}_1,\mathcal{B}_2,\ldots,\mathcal{B}_p)$ is a fixed ordered collection of ordered $p$ blocks for the affine subspaces $U_1,U_2,\ldots,U_m$. Recall the notation $S\coloneqq \cap_{i=1}^m U_i$ and $C_\mathbf{B}$ for the block-wise circumcentered-reflection operator regarding $\mathbf{B}$.  Due to the last auxiliary result, we easily derive linear convergence of Bw-CRM for solving problem \eqref{BestProblem}. Next,  we formally state that  $C_\mathbf{B}$ is a \emph{best approximation mapping} with respect to $S$.

\begin{theorem}[block-wise operator is a BAM]\label{MainTheorem}
Let  $C_\mathbf{B}$ be the block-wise circumcentered-reflection operator regarding $\mathbf{B}$. Then, there exists a constant $r_{\mathbf{B}}\in [0,1)$ so that 
\begin{equation}
\|C_{\mathbf{B}}(z)-P_S(z)\|\leq r_{\mathbf{B}} \|z-P_S(z)\|,
\end{equation}
for all $z\in\RR^n$. Moreover, $P_S(C_{\mathbf{B}}(z))=P_S(z)$ and the convergence of $(C^k_{\mathbf{B}}(z))_{k\in \NN}$ is linear to the unique solution $P_S(z)$, i.e.,
\[\lim_{k\to \infty}C^k_{\mathbf{B}}(z)=P_S(z).\]
Furthermore,  the global Q-linear rate is $r_{\mathbf{B}}\in [0,1)$, i.e.,
for all $k\in\NN$, 
\[\|C^k_{\mathbf{B}}(z)-P_S(z)\|\leq r^k_{\mathbf{B}} \|z-P_S(z)\|.\]
\end{theorem}
\begin{proof} Due to \Cref{def:bwcrm}, $C_{\mathbf{B}}$ is a composition of circumcenter operators, which of each is a BAM (\Cref{thm:BAM}) and thus, by  \Cref{Teo2.3}, it is itself a BAM. The  claims on the sequence  $(C^k_{\mathbf{B}}(z))_{k\in \NN}$ follow then from \Cref{prop:BAMsequence}.
\qed\end{proof}

\subsection{Connections between Bw-CRM and  MAP}
\label{sec:bwCRM-MAP}
Based on our papers \cite{Behling:2017da,Behling:2017bz} and on the previous results, we briefly discuss now some curious connections between Bw-CRM and the method of alternating projections (MAP). 

Our concept of \emph{best approximation mapping} is, by definition,  a relaxation of a projection operator.  With that said, the first relation between Bw-CRM and MAP we want to point out is that Bw-CRM happens to be a \emph{best approximation mapping}, as proven in the last section. Furthermore,  the well known linear convergence of MAP   for a finite number of intersecting affine subspaces~\cite[Theorems~9.31 and 9.33]{Deutsch:2001fl} follows as an immediate consequence of the result on best approximation mappings  stated in  \Cref{thm:BAM}.

Another connection between Bw-CRM and MAP follows from the fact that the projection of a point onto a closed convex set can be seen as the circumcenter regarding the given point and its reflection onto the corresponding set. In other words, if you have a point $z\in\RR^n$ and a closed convex set $U$, then $P_U(z)=\textrm{circumcenter}\{z,R_U(z)\}$ because $P_U(z)\in\aff\{z,R_U(z)\}$ and $\|z-P_U(z)\|=\|R_U(z)-P_U(z)\|$. Therefore, considering the notation from the previous section, we can observe that when all blocks $\mathcal{B}_i$'s have cardinality $1$, \emph{i.e.}, $p=m$ and $\mathcal{B}_i=(U_i)$ for all $i=1,\ldots,m$, we have that $C_{\mathcal{B}_i}$ is precisely the orthogonal projector onto $U_i$. Hence, the block-wise circumcentered-reflection operator $C_{\mathbf{B}}\coloneqq C_{\mathcal{B}_p} \circ \cdots \circ C_{\mathcal{B}_2}\circ C_{\mathcal{B}_1}$ coincides with the MAP operator $P_{U_m} \circ \cdots \circ P_{U_2}\circ P_{U_1}$.

In addition to having the aforementioned connections to MAP, we will see next that the full-block Bw-CRM, \emph{i.e.}, CRM itself, serves as a projector when the multi-set intersection regards only hyperplanes. CRM indeed finds the projection of any given point onto the intersection of hyperplanes in one single step. Perhaps, such a feature might be useful in the implementation of projection methods.

\section{One step convergence of CRM for hyperplane intersection}
\label{sec:onestepCRM-hyperplane}

The initial motivation in the development of our first circumcenter scheme in \cite{Behling:2017da} was defining a method that could handle the trivial problem of finding the intersection of two crossing lines in $\RR^2$ in one step. In the present section, this is done in dimension $n$ for hyperplanes.

The key ingredient that enables the full block Bw-CRM (original CRM) to converge in only one step for hyperplane intersection is that the orthogonal subspace to a given nonempty hyperplane always has dimension one. Interestingly, the first clues on this one-step convergence were indicated by our numerical experiments. Thanks to them we came up with the following results.

\begin{lemma}[one step convergence for full block Bw-CRM]\label{conv-1-step} Consider $\mathcal{H}=(H_1,H_2,\ldots,H_p)$  where $H_i$'s are hyperplanes with nonempty intersection $S_\mathcal{H}$ and let $C_\mathcal{H}$ be the CRM operator regarding $\mathcal{H}$. If $z\in\RR^n$ is so that for all $i=1,2,\ldots,p$ we have $R_{H_i}\cdots R_{H_1}(z)\notin H_i$, then the circumcenter $C_\mathcal{H}(z)$ is already the projection of $z$ onto $S_\mathcal{H}$.
\end{lemma}
\begin{proof}  Without loss of generality, we assume that $H_i$'s are subspaces, as their intersection $S_\mathcal{H}$ is nonempty.

It was proven in \cite[Lemma 3.1]{Behling:2017bz} that $C_\mathcal{H}(z)$ is precisely the projection of $P_{S_\mathcal{H}}(z)$ onto 
\[
W_z= \aff\{z,R_{H_1}(z),R_{H_2}R_{H_1}(z),\ldots ,R_{H_p}\cdots R_{H_2}R_{H_1}(z)\}.
\]
Therefore, by considering the subspace  $\hat{W}_{z}\coloneqq W_{z} - C_\mathcal{H}(z)$, we have
\begin{equation}\label{eq:conv-1-step-perp1}
P_{S_\mathcal{H}}(z) - C_\mathcal{H}(z)\perp \hat{W}_{z}.
\end{equation}
Let $v_1\coloneqq R_{H_1}(z)-z$, $v_2\coloneqq R_{H_2}R_{H_1}(z) - R_{H_1}(z), \ldots$,  $v_p\coloneqq R_{H_p}\cdots R_{H_2}R_{H_1}(z) - R_{H_{p-1}}\cdots R_{H_2}R_{H_1}(z)$. Clearly, $v_{i}\in \hat{W}_{z}$, for $i=1,\ldots, p$ and $\hat{W}_{z} = \lspan\{v_{1},v_{2},\ldots, v_{p}\}$. Also, from the definition of reflection,  we have $v_{i}\perp  H_{i}$, for all $i=1,\ldots, p$. 

By taking into account the hypothesis $R_{H_i}\cdots R_{H_1}(z)\notin H_i$, it is straightforward to conclude that all $v_{i}$'s are not zero. Then, since each $H_{i}$ is a hyperplane, we have 
\[
\lspan\{v_{i}\} = H_{i}^{\perp},\; i=1,\ldots,p.
\]
Now, linear algebra gives us 
\[\begin{aligned}
\hat{W}_{z} & = \lspan\{v_{1},v_{2},\ldots,v_{p}\} \\
			& = \lspan\{v_{1}\} + \lspan\{v_{2}\} + \cdots + \lspan\{v_{p}\} \\ 
			& = H_{1}^{\perp} + H_{2}^{\perp} + \cdots + H_{p}^{\perp} \\
			& = \lspan\{H_{1}^{\perp} \cup H_{2}^{\perp} \cup \cdots \cup H_{p}^{\perp}\}  \\ 
			& = \left(H_{1} \cap H_{2} \cap \cdots \cap H_{p}\right)^{\perp}\\
			& = S_\mathcal{H}^{\perp}
\end{aligned}\]
and from \eqref{eq:conv-1-step-perp1} we have
\begin{equation}\label{eq:conv-1-step-perp2}
P_{S_\mathcal{H}}(z) - C_\mathcal{H}(z)\perp S_\mathcal{H}^{\perp}.
\end{equation}
We have shown in \cite{Behling:2017bz,Behling:2017da} that  $P_{S_\mathcal{H}}(C_\mathcal{H}(z)) = P_{S_\mathcal{H}}(z)$ and because $P_{S_\mathcal{H}}(C_\mathcal{H}(z)) - C_\mathcal{H}(z)$ 	is orthogonal to $S_\mathcal{H}$, it follows that
\begin{equation}\label{eq:conv-1-step-perp3}
P_{S_\mathcal{H}}(z) - C_\mathcal{H}(z)\perp S_\mathcal{H}.
\end{equation}
The combination of \eqref{eq:conv-1-step-perp2} and \eqref{eq:conv-1-step-perp3} implies that $P_{S_\mathcal{H}}(z) - C_\mathcal{H}(z) = 0,$
that is, 
\[
 C_\mathcal{H}(z) = P_{S_\mathcal{H}}(z),
\]
which completes the  proof. \qed\end{proof}

We observe that one can easily construct an example with two lines playing the role of hyperplanes in $\RR^2$ violating the hypothesis in Lemma \ref{conv-1-step} for certain initial points, where indeed the one step convergence of CRM is lost. We might then ask if at least finite convergence of CRM can always be expected in the case of hyperplane intersection. Although we lean towards a positive answer to this interesting theoretical question, we note that it is essentially irrelevant. There are at least two reasons for that. The first is that violating $R_{H_i} \cdots R_{H_1}(z)\notin H_i$ is completely ``{\emph{bad luck}}''.   More formally, one can actually show that the set $ \{z\in\RR^n \mid R_{H_i} \cdots R_{H_1}(z)\notin H_i, \forall i=1,2,\ldots,p\}$ 
is dense in $\RR^n$ (see further comments at the end of the section). The second reason why having $z$ in the complement of the previous set, namely \emph{bad luck},  is not really an issue, is that we can derive a simple and cheap procedure to rewrite our best approximation problem in an equivalent way such that CRM solves the reformulation in one single step. Next we describe this procedure upon a lemma.

\begin{lemma}[procedure for dealing with \emph{bad luck}]\label{lemma:procedure} Consider $\mathcal{H}=(H_1,H_2,\ldots,H_p)$ where the $H_i$'s are hyperplanes with nonempty intersection $S_\mathcal{H}$. Let $z\in\RR^n$ and assume the existence of a smallest index $\check \imath$ in $\{1,2,\ldots,p\}$ for which $\check z\coloneqq  R_{H_{\check \imath}} \cdots R_{H_2}R_{H_1}(z)\in H_{\check \imath}$. Denote by $\check a$ any given non-null orthogonal vector to the hyperplane $H_{\check \imath}$ and let us write $z_{rep}\coloneqq z+tR_{H_1}R_{H_2}\cdots R_{H_{\check\imath-2}}R_{H_{\check \imath-1}}R_{H_{\check \imath}}(\check a)$, where ``\emph{rep}'' stands for the idea of replacement of $z$. Then, for all real number $t$ we have $P_{S_{\mathcal{H}}}(z_{rep})=P_{S_{\mathcal{H}}}(z)$ and for all non-null $t$ sufficiently close to zero it holds that $R_{H_i} \cdots R_{H_2}R_{H_1}(z_{rep})\notin H_{i}$ for $i=1,\ldots,\check\imath$.
\end{lemma}

\begin{proof}
Without loss of generality, assume that the hyperplanes $H_1,H_2,\ldots,H_p$ are subspaces, as their intersection $S_\mathcal{H}$ is nonempty. The fact that reflections onto subspaces preserve the correspondent best approximation solution is a trivial consequence of Pythagoras  and the definition and affinity of the reflections. So, the projections onto $S_{\mathcal{H}}$ of all the points $R_{H_{i}} \cdots R_{H_2}R_{H_1}(z)$ with $i=1,2,\ldots,p$ is given by $P_{S_{\mathcal{H}}}(z)$. This holds in particular for $\check z$. By construction, $t\check a$ is orthogonal to $H_{\check \imath}$, hence we conclude using Pythagoras again that for all real number $t$ the projection of $\check z+t\check a$ onto $S_{\mathcal{H}}$ is also given by $P_{S_{\mathcal{H}}}(z)$. Now, it is easy to see that $z_{rep}$ is defined by reflections of $\check z+t\check a$ onto $H_i$'s starting backwards from the index $\check \imath$ until $1$. Indeed, remind that $z_{rep}\coloneqq z+tR_{H_1}R_{H_2}\cdots R_{H_{\check \imath-2}}R_{H_{\check \imath-1}}R_{H_{\check \imath}}(\check a)$, thus
\[R_{H_1}(z_{rep})=R_{H_1}(z+tR_{H_1}R_{H_2}\cdots R_{H_{\check \imath-2}}R_{H_{\check \imath-1}}R_{H_{\check \imath}}(\check a))\]
Using the linearity of the reflection $R_{H_1}$ and the fact that $R_{H_1}R_{H_1}=\Id$, we get
\[R_{H_1}(z_{rep})=R_{H_1}(z)+tR_{H_2}\cdots R_{H_{\check \imath-2}}R_{H_{\check \imath-1}}R_{H_{\check \imath}}(\check a).\]
Employing this argument successively for $R_{H_2}$ until $R_{H_{\check \imath}}$ implies that
\[R_{H_{\check \imath}}\cdots R_{H_2}R_{H_1}(z_{rep})=R_{H_{\check \imath}}\cdots R_{H_2}R_{H_1}(z)+t\check a,\]
that is,
\[R_{H_{\check \imath}}\cdots R_{H_2}R_{H_1}(z_{rep})=\check z+t\check a.\]
It follows that the projections of $z_{rep}$ and $\check z+t\check a$ onto $S_{\mathcal{H}}$ must coincide. Hence, $P_{S_{\mathcal{H}}}(z_{rep})=P_{S_{\mathcal{H}}}(z)$.

For all non-null $t$ we have $R_{H_{\check \imath}}\cdots R_{H_2}R_{H_1}(z_{rep})=\check z+t\check a\notin H_{\check \imath}$ as $\check a$ is non-null and orthogonal to $H_{\check \imath}$. This gives the lemma if $\check  \imath =1$. So, assume from now on that $\check  \imath>1$.  It remains to show that $R_{H_i} \cdots R_{H_2}R_{H_1}(z_{rep})\notin H_{i}$ for $i=1,\ldots,\check \imath-1$ if we take a non-null $t$ with sufficiently small modulus. That follows easily by hypothesis together with continuity of reflections and Euclidean distance to hyperplanes. By the definition of $\check \imath$ we have that $\dist(R_{H_1}(z),H_1)\neq 0$. Therefore, by continuity in $t$ of the function $$f_1(t)\coloneqq \dist(R_{H_1}(z_{rep}),H_1)=\dist(R_{H_1}(z+tR_{H_1}R_{H_2}\cdots R_{H_{\check \imath-2}}R_{H_{\check \imath-1}}R_{H_{\check \imath}}(\check a)),H_1)$$ we must have a whole interval $[-t_1,t_1]$, with $t_1>0$ for which $f_1(t)\neq 0$. Intervals $[-t_i,t_i]$ with $t_i>0$ like the previous one can be derived in the same way for the remaining indices $i=2,\ldots,\check \imath-1$ by considering the functions $$f_i(t)\coloneqq \dist(R_{H_i}\cdots R_{H_2}R_{H_1}(z_{rep}),H_i)=\dist(R_{H_i}\cdots R_{H_2}R_{H_1}(z+tR_{H_1}R_{H_2}\cdots R_{H_{\check \imath-2}}R_{H_{\check \imath-1}}R_{H_{\check \imath}}(\check a)),H_i).$$ Let $[-\check t,\check t]$ represent the smallest of these intervals. We then have that $R_{H_i} \cdots R_{H_2}R_{H_1}(z_{rep})\notin H_{i}$ for $i=1,\ldots,\check \imath$ if $z_{rep}$ is defined by means of a parameter $t$ belonging to $[-\check t,\check t]$.
\qed

\end{proof}

Note that the previous lemma does not necessarily lead us to a point $z_{rep}$  under the conditions of \Cref{conv-1-step}, we only have an improvement with respect to the index $\check \imath$. Nevertheless, if the $rep$ operation defined in  \Cref{lemma:procedure} is applied successively at most $p-\check \imath$ times, we get a new initial point say $z_{REP}$ so that $P_{S_{\mathcal{H}}}(z)=P_{S_{\mathcal{H}}}(z_{REP})$ and we have $R_{H_i} \cdots R_{H_2}R_{H_1}(z_{REP})\notin H_i$ for all $i=1,2,\ldots,p$. That is, $z_{REP}$ satisfies the conditions of Lemma \ref{conv-1-step} while keeping $P_{S_{\mathcal{H}}}(z)$ as the best approximation solution. This means that the full block Bw-CRM, which is the original CRM, is categorically always able to find the solution of the best approximation problem \eqref{BestProblem} in one single step for hyperplane intersection. Let us state this as a theorem.

\begin{theorem}[one step convergence of CRM]\label{CRM-1-step} Let $\mathcal{H}=(H_1,H_2,\ldots,H_p)$  where $H_i$'s are hyperplanes with nonempty intersection $S_\mathcal{H}$, $C_\mathcal{H}$ be the CRM operator regarding $\mathcal{H}$ and $z\in\RR^n$ be given. Then, CRM finds the projection of $z$ onto $S_\mathcal{H}$ in one single step (with eventual use of $z_{REP}$ as described above).
\end{theorem}

We remind that the probability of having to employ the $rep$ procedure is zero. This is due to the fact that the set of points $z\in\RR^n$ so that $R_{H_i} \cdots R_{H_1}(z)\notin H_i$ for all $i=1,2,\ldots,p$ is dense in $\RR^n$. The density holds because any $z\in\RR^n$ violating the aforementioned conditions can be approximated by a sequence of correspondent $z_{rep}$'s coming from sufficiently shrinking the size of $t\neq 0$ from  \Cref{lemma:procedure}. In any case, note that the $rep$ procedure is implementable. One only needs to consider a backtracking search on the parameter $t$, reflect onto hyperplanes (which can be done by closed formula) and check pertinence to these hyperplanes. 

To finalize the discussion in this section, we would like to present some further remarks.

We want to note that one can consider trivial examples showing that the conditions for one-step convergence in Lemma \ref{conv-1-step}, although sufficient, are not necessary. CRM will converge in one single step whenever the successive reflections generate an affine space of dimension $n-r$, where $r$ is the dimension of the intersection of the given subspaces. One could have the dimension $n-r$ even if the given subspaces are not hyperplanes and also under the \emph{bad luck} of getting reflected points precisely on them.

Our last remark is on possible finite convergence of CRM for hyperplane intersection without employing the $rep$ procedure at all. Although omitting the proof,  we notice  that CRM  converges in at most $3$ steps with no $rep$ procedure for the intersection of $2$ hyperplanes in $\RR^n$.  The challenging question for more than $2$ hyperplanes is left open. Also, we intend to investigate under which conditions one has finite convergence for  CRM, when the subspaces are not all hyperplanes.

\section{Numerical illustrations}
\label{sec:numerical-llustrations}
The geometric nature of Bw-CRM can be used as a tool for solving some classical problems, \emph{e.g.}, the least squares problem, the minimum-norm least-squares (rank deficient) problems, the least-norm solutions of undetermined system and under-determined large-scale linear systems, which are particular instances of problem \eqref{BestProblem}. In this section, we illustrate the  performance of Bw-CRM to solve two related problems: an application in computed tomography and the minimum-norm least square problem. We run all the numerical experiments in \texttt{Julia} language~\cite{Bezanson:2017g}.

\subsection{Application in Computed Tomography}

Reconstruction of images in Computed Tomography (CT) can be addressed by approximately solving linear systems of equations coming from the discretization of suitable inverse problems. Algebraic reconstruction techniques (ART), which are basically MAP type methods, are usually employed to solve those linear systems as not much accuracy is needed  for a solution representing a reasonable image for medical purposes~\cite[Chapter 11]{Herman:2009ej}.  

In this subsection, we solve a problem $As=b$, whose solution provides the  well known Shepp-Logan phantom head~\cite{Shepp:ia}. This  is a standard synthetic  image that serves as the model of a human head and is used for testing  image reconstruction algorithms.   The  data for the matrix $A$ and the vector $b$  were generated using \texttt{AIR Tools II}, a  package by Hansen and J{\o}rgensen~\cite{Hansen:2017ki}, and imported to be used in the \texttt{Julia} implementation. In this case, $A$ has \num{5732} rows and \num{2500} columns. The package also provides the exact $\num{50}\times\num{50}$ pixel Shepp-Logan image, which is represented  as  $ \hat z\in \RR^{2500}$.  

In our experiments, we use Bw-CRM and look at the quality of image reconstructions after a fixed budget of \num{10} iterations. The affine subspaces under consideration are the hyperplanes given by each row of $As=b$. These affine subspaces are distributed in blocks, where each block contains $q$ hyperplanes, except maybe for the last one which contains $(5732\bmod {q})$ hyperplanes.  We exhibit in Table \ref{tab:br-crm_CT} the residue and distance to the actual solution of each version of Bw-CRM, where {Bw-CRM-$q$} indicates that the block size used is ${q}$  --- or  $(5732\bmod {q})$ and the time in seconds of which method. Remind that {Bw-CRM-1} is MAP. It is worth noting that Bw-CRM-16, Bw-CRM-64 and Bw-CRM-256 all beat Bw-CRM-1 (MAP) both in iterations to achieve the same residue.

\begin{table}[H]
\centering
\caption{Bw-CRM applied to CT --  Matrix size: $5732\times 2500$ -- Budget of 10 iterations.} \label{tab:br-crm_CT}
\begin{tabular}
{
l
S[table-format = 1.4e+1]
S[table-format = 1.4e+1]
S[table-format = 1.4e+1]
}
\toprule
{\normalfont \bfseries Method-Block size}     &   {$\|Az_{10}-b\|$} &   {$\|z_{10}- \hat z\|$} &  \textbf{CPU} (s)  \\  
\midrule
Bw-CRM-1 (MAP)   & 3.0321e+01 & 1.3816e+00 & 5.3876 \\
Bw-CRM-16  & 2.8590e+01 & 1.3382e+00       & 8.5242 \\
Bw-CRM-64  & 4.2602e+00 & 1.0332e+00       & 6.1665 \\
Bw-CRM-256 & 7.1039e-01 & 2.7423e-01       & 8.7073 \\
\bottomrule
\end{tabular}
\end{table}
In Figure \ref{fig:CT-reconstructions} we display the original solution and each reconstruction by Bw-CRM for ${q} = 1, 16, 64, 256$. The best solution is achieved by {Bw-CRM-256} at the price of solving $\num{22}$ symmetric positive definite linear systems of size \num{256} and \num{1} of size \num{100}, as $5732 = 22\cdot 256 + 100$.
\begin{figure}[H]\centering
\begin{subfigure}[t]{.3\textwidth}
\centering
\includegraphics[width=\textwidth]{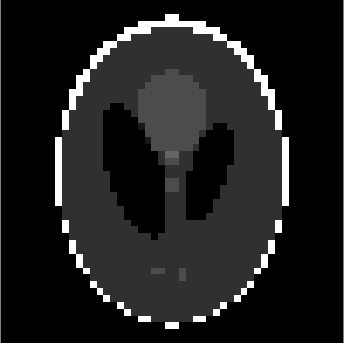}
   \caption{Exact Shepp-Logan}
  \label{fig:phantom-sol}
\end{subfigure} \qquad
\begin{subfigure}[t]{.3\textwidth}
\centering
\includegraphics[width=\textwidth]{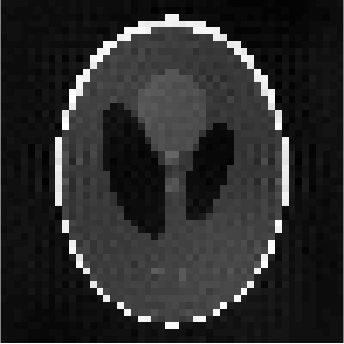}
   \caption{Bw-CRM-1 (MAP)}
  \label{fig:phantom-x1} 
\end{subfigure} \qquad
   \begin{subfigure}[t]{.3\textwidth}
\centering
\includegraphics[width=\textwidth]{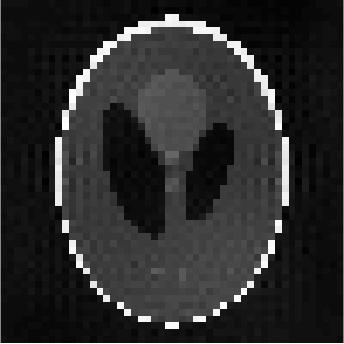}
   \caption{Bw-CRM-16}
  \label{fig:phantom-x16}
\end{subfigure} 
\\[10pt]
\begin{subfigure}[t]{.3\textwidth}
\centering
\includegraphics[width=\textwidth]{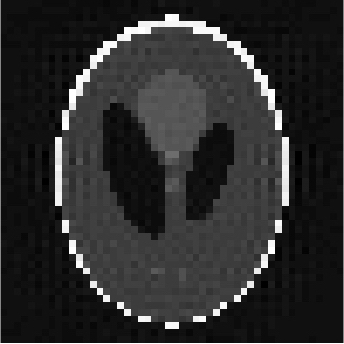}
   \caption{Bw-CRM-64}
  \label{fig:phantom-x64}  
\end{subfigure} \qquad
\begin{subfigure}[t]{.3\textwidth}
\centering
\includegraphics[width=\textwidth]{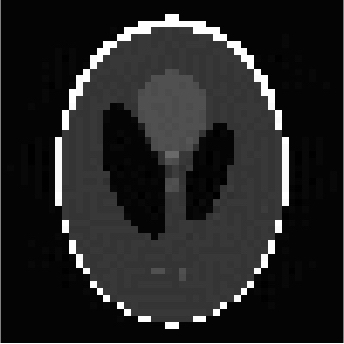}
   \caption{Bw-CRM-256}
  \label{fig:phantom-x256}  
\end{subfigure} 
\caption{CT image reconstructions of Shepp-Logan phantom of size $50\times50$.}
\label{fig:CT-reconstructions}
\end{figure}

\subsection{Solving a least norm problem}

A direct application of Bw-CRM is to solve the following optimization problem: 
Find  $\hat z\in \RR^n$, the solution of 
\begin{equation}
\label{lsp}
\min \|z-s\|, \text{ subject to } As=b,
\end{equation}
where $A\in\RR^{p\times n} (p\leq n)$, $b\in \RR^p$ and $z$ a given vector. The solution $\hat z$ is the closest point to $z$ that lies  in the intersection $S_\mathcal{H}$ of the hyperplanes in $\mathcal{H}\coloneqq ({H}_1, {H}_2, \ldots, {H}_p)$, where  ${H}_i$ is given by the solutions of the $i$-th equation of  $As=b$,   that is, $\hat z $ is the projection of $z$ onto  $S_\mathcal{H}$.

As shown in Section \ref{sec:onestepCRM-hyperplane},  Bw-CRM, when applied to solving this problem by taking the $p$ individual hyperplanes forming the equations (as the main block $\mathcal{H}$),   finds the solution $\hat z$ in just one iteration --- hatring some \emph{bad luck}, as already discussed. If we set $z = 0$, thus $C_{\mathcal{H}}(0)=\hat z$ and problem \eqref{lsp} becomes the minimum norm of under-determined system problem (MNP). It is well-known that if $A$ has full rank we can solve \eqref{lsp} by using the Moore-Penrose pseudo-inverse of $A$, as $\hat z= z +  A^T(AA^T)^{-1}(b - Az)$.

In order to illustrate various possible choices of blocks for Bw-CRM, we solve problem \eqref{lsp} using matrix coming from a finite element modeling, called \texttt{FIDAP005}, and available at Matrix Market~\cite{Boisvert:1997gt}. The matrix $A$ is given by selecting respectively the first 12, 24 and 27  rows of \texttt{FIDAP005}, $b$ is the correspondent vector of ones and we take $z= 0$.  The structure of the entire sparse matrix \texttt{FIDAP005} is shown in Figure~\ref{fig:FIDAP005}.
\begin{figure}[!ht]
\centering
\includegraphics[width=.45\textwidth]{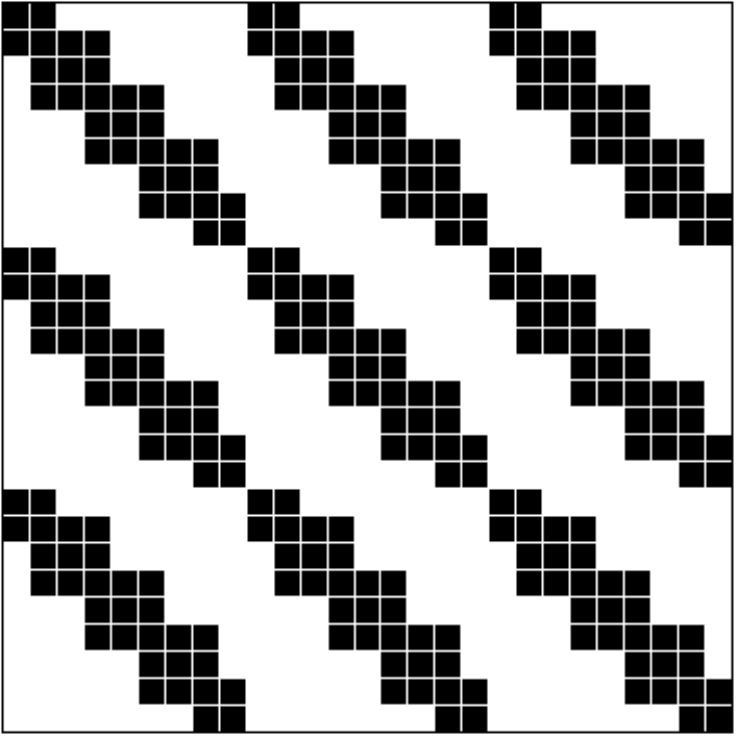}
\caption{Matrix \texttt{FIDAP005} sparsity structure.}
\label{fig:FIDAP005}
\end{figure}
Next, we show the results for Bw-CRM in  \Cref{tab:fidap005-12x27,tab:fidap005-24x27,tab:fidap005-27x27}, where each subspace under consideration is  given by a row equation of $As=b$. The different size of block choices are displayed in the first column of the tables, followed by the number of blocks, the number of projections/reflections, the number of iterations, the norm of the residue and the CPU time, in seconds. The stopping criterion was having the norm of the residue smaller than the labeled tolerance \texttt{tol}. Note that \Cref{tab:fidap005-27x27} presents the results where the sparse block structure of 
matrix \texttt{FIDAP005} is explored.

\begin{table}[htpb]
\centering
\caption{Results for Bw-CRM -- Matrix size: $12\times27$ -- $\tol = \num{e-5}$.} \label{tab:fidap005-12x27}
\begin{tabular}
{
l
S[table-format = 2]
S[table-format = 3]
S[table-format = 2]
S[table-format = 1.4e+2]
S[table-format = 1.4e+1]
}
\toprule
{\normalfont \bfseries Method-Block size}    &    \textbf{Blocks} &    \textbf{Proj/Reflec} &  \textbf{Iter} &   {$\|Az_{\text{Iter}}-b\|$} &   \textbf{CPU} (s) \\ 
\midrule
Bw-CRM-1 (MAP)	& 12 & 180 & 15 & 4.2323e-06 & 6.2563e-04 \\
Bw-CRM-2  		& 6  & 156 & 13 & 7.1863e-06 & 8.7344e-04 \\
Bw-CRM-3  		& 4  & 180 & 15 & 5.2967e-06 & 9.3172e-04 \\
Bw-CRM-4  		& 3  & 120 & 10 & 8.5466e-06 & 5.7481e-04 \\
Bw-CRM-6  		& 2  & 132 & 11 & 3.4566e-06 & 6.1370e-04 \\
Bw-CRM-12 (CRM)	& 1  & 12  & 1  & 7.8280e-14 & 9.9924e-05 \\
\bottomrule
\end{tabular}
\end{table}


\begin{table}[htpb]
\centering
\caption{Results for Bw-CRM -- Matrix size: $24\times27$ --  $\tol = \num{e-5}$. } \label{tab:fidap005-24x27}
\begin{tabular}
{
l
S[table-format = 2]
S[table-format = 5]
S[table-format = 3]
S[table-format = 1.4e2]
S[table-format = 1.4e+1]
}
\toprule
{\normalfont \bfseries Method-Block size}    &    \textbf{Blocks} &    \textbf{Proj/Reflec} &  \textbf{Iter} &   {$\|Az_{\text{Iter}}-b\|$} &   \textbf{CPU} (s) \\ 
\midrule
Bw-CRM-1 (MAP)	& 24 & 12048 & 502 & 9.8869e-06 & 6.9701e-02 \\
Bw-CRM-2  		& 12 & 11904 & 496 & 9.8041e-06 & 1.2950e-01 \\
Bw-CRM-3  		& 8  & 11160 & 465 & 9.9053e-06 & 9.9697e-02 \\
Bw-CRM-4  		& 6  & 9576  & 399 & 9.9067e-06 & 6.5870e-02 \\
Bw-CRM-6  		& 4  & 11880 & 495 & 9.8665e-06 & 8.9289e-02 \\
Bw-CRM-8  		& 3  & 10440 & 435 & 9.7581e-06 & 7.0528e-02 \\
Bw-CRM-12 		& 2  & 8328  & 347 & 9.7967e-06 & 8.1303e-02 \\
Bw-CRM-24 (CRM)	& 1  & 24    & 1   & 1.4852e-12 & 1.4166e-04 \\
\bottomrule
\end{tabular}
\end{table}

\begin{table}[htpb]
\centering
\caption{Results for Bw-CRM -- Matrix size: $27\times27$ --  $\tol = \num{e-3}$. \label{tab:fidap005-27x27}}
\begin{tabular}
{
l
S[table-format = 1]
S[table-format = 6]
S[table-format = 5]
S[table-format = 1.4e+2]
S[table-format = 1.4e+1]
}
\toprule
{\normalfont \bfseries Method-Block size}    &    \textbf{Blocks} &    \textbf{Proj/Reflec} &  \textbf{Iter} &   {$\|Az_{\text{Iter}}-b\|$} &   \textbf{CPU} (s) \\ 
\midrule
BW-CRM-1 (MAP)  & 27 & 3992166 & 147858 & 9.9998e-04 & 3.0667e+01 \\
BW-CRM-3  					  & 9  & 3448602 & 127726 & 9.9997e-04 & 2.6326e+01 \\
BW-CRM-9  					  & 3  & 3209355 & 118865 & 9.9999e-04 & 2.9115e+01 \\
BW-CRM-27 (CRM) & 1  & 27      & 1      & 6.9229e-10 & 1.5987e-04 \\
\bottomrule
\end{tabular}
\end{table}

As expected by the results of Section \ref{sec:onestepCRM-hyperplane}, the full block Bw-CRM converges in one iteration for the hyperplane intersection problems above. Note that we have to be careful when looking at the CPU time as it depends on the inner linear system solver for finding circumcenters. What we can say, though, is that the number of iterations tends to slightly increase as the number of blocks increase. It would be interesting to investigate whether there exists a sort of optimal block size, with respect to particular instances.

In contrast to the feasible set of the problems regarding \Cref{tab:fidap005-12x27,tab:fidap005-24x27}, the feasible set of the problem addressed in  \Cref{tab:fidap005-27x27} reduces to a singleton. Even though it is known that MAP suffers from zig-zag behavior,  we got surprised with the huge amount of iterations that it took to converge in the case of \Cref{tab:fidap005-27x27}. We have established  connections between Bw-CRM and MAP in \Cref{sec:bwCRM-MAP} and unfortunately it seems that, in the case, when MAP performs poorly this is inherited by Bw-CRM, except for the full block Bw-CRM. This is a motivation for future investigation on randomized order of subspaces or blocks for  Bw-CRM, as randomized versions of MAP performs a lot better \cite{Strohmer:2008cm}.

\section{Concluding remarks}\label{sec:conlcuding}

We presented new notions and results regarding circumcenter schemes for projecting a given point onto the (nonempty) intersection of a finite number of affine subspaces. Circumcenter iterations were introduced in \cite{Behling:2017da} and shown to provide a better bond between reflections than the one considered in the classical Douglas-Rachford approach. The results in \cite{Behling:2017bz} improved \cite{Behling:2017da} by enabling the Circumcentered-Reflection Method (CRM) to deal with $m>2$ affine subspaces. In the present article we also dealt with more than two sets. We defined the Block-wise Circumcentered-Reflection Method (Bw-CRM), which considers the $m$ affine subspaces in blocks. More precisely, we composed circumcenter operators along a partition of the indices $1,2,\ldots,m$. In this way, the original circumcenter method from \cite{Behling:2017bz} can be seen as Bw-CRM with one full block, where this block contains all $m$ affine subspaces. It was interesting that by considering Bw-CRM with $m$ blocks, \emph{i.e.}, the case where each block contains exactly one affine subspace, we recovered the famous method of alternating projections (MAP). Linear convergence for any blocks choice of Bw-CRM was proven. Our proof was carried out in a unified fashion thanks to the introduction of a new concept, the one of \emph{best approximation mapping}. In addition to deriving theoretical linear convergence of Bw-CRM, numerical experiments were run. For the numerical tests we considered blocks with homogeneous cardinality in order to investigate the relation between speed of convergence (time/complexity) and number of blocks in Bw-CRM. The experiments also indicated what became a curious result in this paper: it turns out that CRM (Bw-CRM with one full block) finds the projection of any given point onto the intersection of hyperplanes in one single step. 

This work contributed not only with a deeper understanding of circumcenter type methods, we think that our results represent another step towards using circumcenters in other settings. Our future research will be focused on enforcing circumcenter iterations for solving the nonconvex problem: Find $z\in S$ with 
\[S=U\cap V,\] where $ U=\bigcup_{i=1}^{m}U_i$, with $U_i$, for each $i$, being a subspaces and $V$ being an affine subspace. This problem contains as a particular case the nonconvex sparse affine feasibility problem for which DRM and MAP fail to converge globally. We have strong convictions based on initial numerical tests and some preliminary proofs that a (block-wise) circumcenter method can perform very well (global convergence) for this kind of affine-structured problem.

\begin{acknowledgements}
We dedicate this paper in honor of the 70th birthday of Professor J. M. Mart\'inez and of the 60th birthday of Professor Yuan Jinyun. The first author wants to thank the Federal University of Santa Catarina and remarks that part of his contribution to the present work was  carried out at this institution. We thank the anonymous referees for their valuable suggestions which  significantly improved the presentation of this manuscript.
\end{acknowledgements}

\bibliographystyle{spmpsci}

\bibliography{refs}

\begin{thebibliography}{10}
\providecommand{\url}[1]{{#1}}
\providecommand{\urlprefix}{URL }
\expandafter\ifx\csname urlstyle\endcsname\relax
  \providecommand{\doi}[1]{DOI~\discretionary{}{}{}#1}\else
  \providecommand{\doi}{DOI~\discretionary{}{}{}\begingroup
  \urlstyle{rm}\Url}\fi

\bibitem{Artacho:2014uo}
Arag{\'o}n~Artacho, F.J., Borwein, J.M., Tam, M.K.: {Recent Results on
  Douglas{\textendash}Rachford Methods for Combinatorial Optimization
  Problems}.
\newblock J. Optim. Theory Appl. \textbf{163}(1), 1--30 (2013)

\bibitem{AragonArtacho:2019ug}
Arag{\'o}n~Artacho, F.J., Campoy, R., Tam, M.K.: {The Douglas-Rachford
  Algorithm for Convex and Nonconvex Feasibility Problems}.
\newblock arXiv (1904.09148) (2019)

\bibitem{BCNPW14}
Bauschke, H.H., Bello-Cruz, J.Y., Nghia, T.T.A., Phan, H.M., Wang, X.: {The
  rate of linear convergence of the Douglas{\textendash}Rachford algorithm for
  subspaces is the cosine of the Friedrichs angle}.
\newblock J. Approx. Theory \textbf{185}, 63--79 (2014)

\bibitem{BCNPW15}
Bauschke, H.H., Bello-Cruz, J.Y., Nghia, T.T.A., Phan, H.M., Wang, X.: {Optimal
  Rates of Linear Convergence of Relaxed Alternating Projections and
  Generalized Douglas-Rachford Methods for Two Subspaces}.
\newblock Numer. Algorithms \textbf{73}(1), 33--76 (2016)

\bibitem{Bauschke:1993dd}
Bauschke, H.H., Borwein, J.M.: {On the convergence of von Neumann's alternating
  projection algorithm for two sets}.
\newblock Set-Valued Anal. \textbf{1}(2), 185--212 (1993)

\bibitem{Bauschke:2006ej}
Bauschke, H.H., Borwein, J.M.: {On Projection Algorithms for Solving Convex
  Feasibility Problems}.
\newblock SIAM Rev. \textbf{38}(3), 367--426 (2006)

\bibitem{Bauschke:2003gb}
Bauschke, H.H., Deutsch, F.R., Hundal, H., Park, S.H.: {Accelerating the
  Convergence of the Method of Alternating Projections}.
\newblock Trans. Amer. Math. Soc. \textbf{355}(9), 3433--3461 (2003)

\bibitem{Bauschke:2013jb}
Bauschke, H.H., Luke, D.R., Phan, H.M., Wang, X.: {Restricted Normal Cones and
  the Method of Alternating Projections: Theory}.
\newblock Set-Valued Var. Anal. \textbf{21}(3), 431--473 (2013)

\bibitem{Bauschke:2016jw}
Bauschke, H.H., Moursi, W.M.: {The Douglas--Rachford Algorithm for Two (Not
  Necessarily Intersecting) Affine Subspaces}.
\newblock SIAM J. Optim. \textbf{26}(2), 968--985 (2016)

\bibitem{Bauschke:2018ut}
Bauschke, H.H., Ouyang, H., Wang, X.: {On circumcenters of finite sets in
  Hilbert spaces}.
\newblock Linear Nonlinear Anal. \textbf{4}(2), 271--295 (2018)

\bibitem{Bauschke:2019uh}
Bauschke, H.H., Ouyang, H., Wang, X.: {Circumcentered methods induced by
  isometries}.
\newblock arXiv (1908.11576) (2019)

\bibitem{Bauschke:2018wa}
Bauschke, H.H., Ouyang, H., Wang, X.: {On circumcenter mappings induced by
  nonexpansive operators}.
\newblock Pure and Applied Functional Analysis  (in press)

\bibitem{Behling:2017da}
Behling, R., Bello-Cruz, J.Y., Santos, L.R.: {Circumcentering the
  Douglas{\textendash}Rachford method}.
\newblock Numer. Algorithms \textbf{78}(3), 759--776 (2018)

\bibitem{Behling:2017bz}
Behling, R., Bello-Cruz, J.Y., Santos, L.R.: {On the linear convergence of the
  circumcentered-reflection method}.
\newblock Oper. Res. Lett. \textbf{46}(2), 159--162 (2018)

\bibitem{Bezanson:2017g}
Bezanson, J., Edelman, A., Karpinski, S., Shah, V.B.: {Julia: A Fresh Approach
  to Numerical Computing}.
\newblock SIAM Rev. \textbf{59}(1), 65--98 (2017)

\bibitem{Boisvert:1997gt}
Boisvert, R.F., Pozo, R., Remington, K., Barrett, R.F., Dongarra, J.J.: {Matrix
  Market: a web resource for test matrix collections}.
\newblock In: R.F. Boisvert (ed.) Quality of Numerical Software, pp. 125--137.
  Springer, Boston, MA, Boston, MA (1997)

\bibitem{Borwein:2014ka}
Borwein, J.M., Tam, M.K.: {A Cyclic Douglas{\textendash}Rachford Iteration
  Scheme}.
\newblock J. Optim. Theory Appl. \textbf{160}(1), 1--29 (2014)

\bibitem{Borwein:2015vm}
Borwein, J.M., Tam, M.K.: {The cyclic Douglas-Rachford method for inconsistent
  feasibility problems}.
\newblock J. Nonlinear Convex Anal. \textbf{16}(4), 573--584 (2015)

\bibitem{Demanet:2016fj}
Demanet, L., Zhang, X.: {Eventual linear convergence of the Douglas-Rachford
  iteration for basis pursuit}.
\newblock Math. Comp. \textbf{85}(297), 209--238 (2016)

\bibitem{Deutsch:1995ja}
Deutsch, F.R.: {The Angle Between Subspaces of a Hilbert Space}.
\newblock In: S.P. Singh (ed.) Approximation Theory, Wavelets and Applications,
  pp. 107--130. Springer, Dordrecht, Dordrecht (1995)

\bibitem{Deutsch:2001fl}
Deutsch, F.R.: {Best Approximation in Inner Product Spaces}.
\newblock CMS Books in Mathematics. Springer, New York, NY (2001)

\bibitem{Hansen:2017ki}
Hansen, P.C., J{\o}rgensen, J.S.: {AIR Tools II: algebraic iterative
  reconstruction methods, improved implementation}.
\newblock Numer. Algorithms \textbf{79}(1), 107--137 (2017)

\bibitem{Herman:2009ej}
Herman, G.T.: {Fundamentals of computerized tomography}, 2 edn.
\newblock Advances in Pattern Recognition. Springer, Dordrecht (2009)

\bibitem{Hesse:2013cv}
Hesse, R., Luke, D.R.: {Nonconvex Notions of Regularity and Convergence of
  Fundamental Algorithms for Feasibility Problems}.
\newblock SIAM J. Optim. \textbf{23}(4), 2397--2419 (2013)

\bibitem{Hesse:2014gi}
Hesse, R., Luke, D.R., Neumann, P.: {Alternating Projections and
  Douglas-Rachford for Sparse Affine Feasibility}.
\newblock IEEE Trans. Signal Process. \textbf{62}(18), 4868--4881 (2014)

\bibitem{Lindstrom:2018uc}
Lindstrom, S.B., Sims, B.: {Survey: Sixty Years of Douglas--Rachford}.
\newblock arXiv (1809.07181) (2018)

\bibitem{Ouyang:2018gu}
Ouyang, H.: {Circumcenter operators in Hilbert spaces}.
\newblock Master's thesis, University of British Columbia, Okanagan (2018)

\bibitem{Shepp:ia}
Shepp, L.A., Logan, B.F.: {The Fourier reconstruction of a head section}.
\newblock IEEE Trans. Nucl. Sci. \textbf{21}(3), 21--43 (1974)

\bibitem{Strohmer:2008cm}
Strohmer, T., Vershynin, R.: {A Randomized Kaczmarz Algorithm with Exponential
  Convergence}.
\newblock J Fourier Anal Appl \textbf{15}(2), 262--278 (2008)

\end{thebibliography}

\end{document}